\documentclass{amsart}
\usepackage{graphicx}
\usepackage[all]{xy}
\usepackage[ansinew]{inputenc}
\usepackage{amsmath}
\usepackage{amscd}
\usepackage{amssymb}
\usepackage{latexsym}
\usepackage{mathrsfs}
\usepackage{hyperref}

\newtheorem{thm}{Theorem}[section]

\theoremstyle{definition}
\newtheorem{cor}[thm]{Corollary}

\newtheorem{prop}[thm]{Proposition}
\newtheorem{defn}[thm]{Definition}
\newtheorem{fact}[thm]{Fact}

\newtheorem*{remark}{Remark}
\numberwithin{equation}{section}

\newcommand{\N}{\mathbb{N}}
\newcommand{\Z}{\mathbb{Z}}
\newcommand{\Q}{\mathbb{Q}}
\newcommand{\R}{\mathbb{R}}
\newcommand{\C}{\mathbb{C}}

\newcommand{\dcl}{\operatorname{dcl}}

\newcommand{\rk}{\operatorname{rk}}

\newcommand{\fin}[1]{\textrm{Fin}(#1)}
\newcommand{\Un}[1]{\textrm{Un}(#1)}

\newcommand{\Cal}{\mathcal}

\newcommand{\rtil}{\tilde{\mathbb{R}}}

\def \<{\langle}
\def \>{\rangle}

\def \((  {(\!(}
\def \)) {)\!)}

\begin{document}
\bibliographystyle{plain}
\title{Dependent Pairs}

\author{Ayhan G\"unayd\i n, Philipp Hieronymi}
\thanks{The first author was supported by Marie Curie contract MRTN-2004-5512234. The second author was funded by \emph{Deutscher Akademischer Austausch Dienst Doktorandenstipendium} and by the \emph{Engineering and Physical Sciences Research Council}.}%

\address{Centro de Matem\'{a}tica e Aplica\c{c}\~{o}es Fundamentais,
Av. Prof. Gama Pinto, 2,
1649-003 Lisboa, Portugal}

\address{University of Illinois at Urbana-Champaign, Department of Mathematics, 1409 W. Green Street, Urbana, IL 61801, USA }

\email{ayhan@ptmat.fc.ul.pt}

\email{P@hieronymi.de}

\date{\today}
\subjclass[2000]{}


\begin{abstract}
We prove that certain pairs of ordered structures are dependent. Among these structures are dense and tame
pairs of o-minimal structures and further the real field with a multiplicative subgroup with the Mann property, regardless of whether it is dense or discrete.
\end{abstract}

\maketitle

\section{Introduction}
The {\it independence property} was first introduced by Shelah in
\cite{Shelah}. The definition we give below is not the original definition in that paper; for the equivalence of these two definitions and the basics of this property see \cite{poizat}.  As a matter of fact, we define the absence of the independence property for a theory, and we call such a theory {\it dependent} (in some sources it appears as {\it `the theory has NIP'}).

%

Here we construct examples of dependent theories
appearing in a natural way. We show that the theories of dense pairs
of o-minimal expansions of ordered abelian groups and tame pairs of
o-minimal expansions of real closed fields are dependent (see Sections \ref{densepairs} and \ref{tamepairs}). Actually our
techniques apply to more general pairs. For instance, in Subsection \ref{mann} we prove that the theory of a
real closed field $R$ expanded by a dense multiplicative subgroup of $R^{>0}$ with the
Mann property is dependent whenever the group has the property that
for every prime $p$ the subgroup of $p^{\operatorname{th}}$ powers
has finite index.
Finally, in the last section we show that a real closed field expanded by a cyclic multiplicative subgroup generated by a positive element is dependent.

\medskip\noindent
Let $x_1,\dots,x_p,y_1,\dots,y_q$ be distinct variables and put $\vec x=(x_1,\dots,x_p)$ and $\vec y=(y_1,\dots,y_q)$. Here is a precise definition of the property under consideration.

\begin{defn}
Let $T$ be a complete theory in a language $\Cal L$ and let $\mathbb M$ be a monster model of it.

\smallskip\noindent
 (1) We say that an $\Cal L$-formula $\varphi(\vec x,\vec y)$ is {\it dependent} (in $T$) if for every indiscernible sequence
$(\vec a_i)_{i\in \omega}$ from $\mathbb M^p$ and every $\vec{b}\in\mathbb M^q$, there is $i_0\in\omega$ such that either $\mathbb M\models \varphi(\vec a_i,\vec{b})$ for every $i>i_0$  or $\mathbb M\models \neg\varphi(\vec a_i,\vec{b})$ for every $i>i_0$.

\smallskip\noindent
(2)  The theory $T$ is {\it dependent} if every $\Cal L$-formula is dependent in $T$.
\end{defn}

\noindent A key fact from \cite{Shelah} is that a theory is dependent, if all
formulas of the form $\varphi(x,\vec{y})$ are dependent.

\medskip\noindent
Let $\Cal A=(A,<,\dots)$ be an o-minimal structure in a language $\Cal L$ and $B$ is a subset of $A$.  We consider the structure $(\Cal A,B)$ in the language $\Cal L(U):=\Cal L\cup\{U\}$, where $U$ is a unary relation symbol not in $\Cal L$; and $T_B$ denotes the $\Cal L(U)$-theory of $(\Cal A,B)$.
Vaguely speaking, we have two cases according to $B$ being dense or discrete in its convex hull in $A$. We handle these cases separately; see Theorems \ref{nipisoc} and \ref{thmdiscpair} for the general results.

\medskip\noindent
We have learned that Berenstein, Dolich and Onshuus have
been working on similar topics. However, after communicating with
Berenstein, we have decided that there are enough differences between the two projects.

\medskip\noindent
{\it Acknowledgements.} We thank Martin Bays and Oleg Belegradek for helpful comments on earlier versions of this paper.

\medskip\noindent
{\it Notations, conventions.} Throughout $m,n,p,q,r$ range over $\N:=\{0,1,2,\dots\}$, the set of natural numbers, which we distinguish from the first infinite ordinal $\omega$. We usually let $i,j,k$ denote elements of $\omega$.  Also {\it `definable'} means `definable with parameters'; in the case that we want to make parameters explicit, we write $\Cal L$-$B$-definable where $B$ is a subset of the appropriate structure.

\smallskip\noindent
We name model theoretic structures with capital letters in Calligraphic font, and the underlying set of these structures with the same capital letter in the normal font, with the exception that both a monster model of a theory and its underlying set are denoted by a capital letter in blackboard bold font, although we reserve $\N,\Z,\Q,\R,\C$ for their standard use. For instance $\Cal R=(R,\dots)$ will denote an arbitrary model theoretic structure with $R$ as the underlying set, whereas $\mathbb M$ will denote a monster model of a theory and its underlying set.

\smallskip\noindent
We use letters $x,y,z$ for variables and letters $a,b,c$ for elements from the underlying set of a structure. We distinguish tuples of variables from a single variable by using vector notation, likewise for tuples of elements. For example, $x$ is a single variable and $\vec x$ is a tuple of variables.

\section{Fabricating better indiscernible sequences}

We accumulate the technical details in this section. First we give some combinatorial results, which will be useful in the rest of the paper. They are folklore, but we include proofs for completeness. Then we prove a general reduction step for proving pair-like structures are dependent.

\medskip\noindent
Let $T'$ be a theory in the first order language $\Cal L'$ with a monster model $\mathbb M$. As in the Introduction, let $x_1,\dots,x_p,y_1,\dots,y_q$ be distinct variables and put $\vec x=(x_1,\dots,x_p)$ and $\vec y=(y_1,\dots,y_q)$.

\begin{prop}\label{indiscernible1}
 Let  $(\vec a_i)_{i\in\omega}$ be an indiscernible sequence from $\mathbb M^p$. Suppose that $\phi(\vec x,\vec y)$ is an $\Cal L'$-formula such that $\mathbb M\models \exists\vec y\,\phi(\vec a_i,\vec y)$ for some $i\in\omega$. Then there is an indiscernible sequence $(\vec b_i)_{i\in\omega}$ from $\mathbb M^q$ such that $\mathbb M\models \phi(\vec a_i,\vec b_i)$ for every $i\in\omega$.
\end{prop}

\begin{proof}
First note that since $(\vec a_i)_{i\in\omega}$ is indiscernible, we can take $\vec c_i\in\mathbb M^q$ for every $i\in\omega$ such that $\mathbb M\models \phi(\vec a_i,\vec c_i)$.

\medskip\noindent
Let $\vec z=(\vec z_i)_{i\in\omega}$ be a countable tuple of distinct tuples of variables of length $q$ and let $\Sigma(\vec z)$ be the collection of $\Cal L'$-formulas of the form:
\begin{equation}\label{1}
 \phi(\vec a_i,\vec z_i)\text{ or }
\end{equation}
\begin{equation}\label{2}
 \psi(\vec z_{i_1},\dots,\vec z_{i_m})\leftrightarrow\psi(\vec z_{j_1},\dots,\vec z_{j_m}),
\end{equation}
where $\psi$ is an $\Cal L'$-formula and $i_1<\dots<i_m<\omega$ and $j_1<\dots<j_m<\omega$. We need to show that $\Sigma(\vec z)$ is consistent. By saturation it suffices to show that every finite subset $\Sigma'$ of $\Sigma(\vec z)$ is realized in $\mathbb M$.

\medskip\noindent
Take $k\in\omega$ and a finite set $\Delta$ of  $\Cal L'$-formulas such that if
$$\psi(\vec z_{i_1},\dots,\vec z_{i_m})\leftrightarrow\psi(\vec z_{j_1},\dots,\vec z_{j_m})$$
is in $\Sigma'$ then $i_1<\dots<i_m<k$, $j_1<\dots<j_m<k$ and $\psi\in\Delta$.

\medskip\noindent
For $i_1<\dots<i_m$ put
$$X(i_1,\dots,i_m):=\{\psi\in\Delta:\mathbb M\models\psi(\vec c_{i_1},\dots,\vec c_{i_m})\}.$$
By Ramsey's Theorem there is an infinite subset $S$ of $\omega$ such that for every $i_1<\dots<i_m$ and $j_1<\dots<j_m$ from $S$ we have
$$X(i_1,\dots,i_m)=X(j_1,\dots,j_m).$$
Take a subset $I$ of $S$ of cardinality $k$ and let $\theta\big((\vec a_i)_{i\in I},(\vec z_i)_{i\in I}\big)$ be the conjunction of formulas
$$\phi(\vec a_i,\vec z_i)\wedge\psi(\vec z_{i_1},\dots,\vec z_{i_m})\leftrightarrow\psi(\vec z_{j_1},\dots,\vec z_{j_m}),$$
with $i\in I$, $i_1<\dots<i_m$ and $j_1<\dots<j_m$ are varying in $I$ and $\psi$ varying in $\Delta$. Since $\mathbb M\models\theta\big((\vec a_i)_{i\in I},(\vec c_i)_{i\in I}\big)$ we have
$$\mathbb M\models\exists \vec y_1\cdots\exists \vec y_k\, \theta\big((\vec a_i)_{i\in I},(\vec y_1,\dots,\vec y_k)\big),$$
where $\vec y_1,\dots,\vec y_k$ are tuples of distinct variables of length $q$. Then by the indiscerniblity of $(\vec a_i)_{i\in\omega}$, it follows that $\Sigma'$ is consistent.
Therefore so is $\Sigma$.
\end{proof}

\begin{prop}\label{shelahprop} Let $(\vec a_i)_{i \in \omega}$ be an indiscernible sequence from $\mathbb M^p$. Also let $\vec{b} \in \mathbb M^q$ and $\psi(\vec{x},\vec{y})$ be an $\mathcal L'$-formula such that both $\{ i \in \omega : \mathbb M \models \psi(\vec a_i,\vec{b})\}$ and $\{ i \in \omega : \mathbb M \models \neg\psi(\vec a_i,\vec{b})\}$ are infinite. Then there is a sequence $(\vec c_i)_{i\in \omega}$ such that
\begin{itemize}
\item $(\vec c_{2i})_{i\in \omega}$ is indiscernible over $\vec{b}$,
\item $\mathbb{M} \models \psi(\vec c_i)$ if and only if $i$ is even, and
\item $\mathbb{M}\models \varphi(\vec c_1,\dots,\vec c_k)$
 if and only if $\mathbb{M}\models \varphi(\vec a_{1},\dots,\vec a_{k})$ for every $\Cal L'$-formula $\varphi$.
\end{itemize}
\end{prop}
\begin{proof} We may assume that $\mathbb M \models \psi(\vec a_i,\vec{b})$ if and only if $i$ is even. Let $\vec z=(\vec z_i)_{i\in\omega}$ be a countable tuple of distinct tuples of variables of length $p$ and let $P=P(\vec z)$ be the type of $(\vec{a_i})_{i\in \omega}$. Let $\Delta$ be the following set of $\mathcal L'$-formulas
$$
\{ \psi(\vec z_{2i},\vec{b}),\neg \psi(\vec z_{2i+1},\vec{b}) : i \in \omega\}.
$$
Further, for an $\mathcal L'$-formula $\varphi(\vec u_1,\dots,\vec u_n,\vec{y})$ we denote by $\Sigma^{\varphi}$ the set of all $\mathcal L'$-formulas of the form
$$
\varphi(\vec z_{2i_1},\dots,\vec z_{2i_n},\vec{b}) \leftrightarrow \varphi(\vec z_{2j_1},\dots,\vec z_{2j_n},\vec{b}),
$$
where $i_1<\dots<i_n$ and $j_1<\dots<j_n$.\\
\noindent It is just left show that $$P \cup \Delta \cup \bigcup \{ \Sigma^{\varphi} : \varphi \textrm{ is an $\mathcal L'$-formula}\}$$ is finitely realizable in $\mathbb M$. Let $\varphi_1,\dots,\varphi_m$ be $\mathcal L'$-formulas. By Ramsey's theorem, there is an infinite subset $S$ of $\omega$ such that for every $j=1,\dots,m$ and $s_1<\dots<s_n$ and $t_1<\dots<t_n$ from $S$ we have
$$\varphi_j(\vec a_{2s_1},\dots,\vec a_{2s_n},\vec{b})\leftrightarrow\varphi_j(\vec a_{2t_1},\dots,\vec a_{2t_n},\vec{b}).$$
Put $S':=\{2i:i\in S\} \cup\{2i+1:i\in S\}$. Then $(a_{i})_{S'}$ realizes
$$P \cup \Delta \cup \bigcup \{ \Sigma^{\varphi_j} : j=1,\dots,m\}.$$
\end{proof}

\medskip\noindent
In the rest of this section let $\Cal L$ be a first order language and $\Cal L(U):=\Cal L\cup\{U\}$ where $U$ is a unary predicate not in $\Cal L$.  Let $T_U$ be an $\Cal L(U)$-theory with a monster model $\mathbb M$ and let $\dcl_{\Cal L}$ denote the definable closure in the $\Cal L$-reduct. In this setting we have the following as a consequence of Proposition \ref{indiscernible1} by taking $\Cal L'=\Cal L(U)$ and $\phi$ to be the appropriate $\Cal L(U)$-formula witnessing that $a_i\in\dcl_{\Cal L}(U(\mathbb M)\cup\{a_1,\dots,a_{i_0}\})$.

 \begin{prop}\label{indlemma}
Let $(a_i)_{i\in\omega}$ be an indiscernible sequence such that there are $i_0<i\in\omega$ with $a_i\in\dcl_\Cal L\big(U(\mathbb
M)\cup\{a_1,\dots,a_{i_0}\}\big)$. Then there exist a
function $f:\mathbb M^m\to\mathbb M$ definable in $\mathbb M$ over
$a_1,\dots,a_{i_0}$ and an indiscernible sequence $(\vec g_i)_{i>i_0}$
such that for every $i>i_0$ we have $\vec g_i \in U(\mathbb
M)^m$ and $f(\vec g_i)=a_i$.
\end{prop}

\medskip\noindent
For the next result suppose that $\dcl_{\Cal L}$ is a pregeometry. This way we get a notion of {\it rank of $M$ over $N$}, for subsets $M,N$ of $\mathbb M$.  More precisely
$$\rk(M|N):=\inf\big\{|X|:X\subseteq M\text{ and} \dcl_{\Cal L}(X\cup N)=\dcl_{\Cal L}(M\cup N)\big\}.$$
Note that $\rk(M|N)$ can be infinite.
\begin{prop}\label{mainthm} Suppose that
\begin{enumerate}
 \item for every formula $\varphi(\vec x,\vec y)$, indiscernible sequence $(\vec g_i)_{i\in\omega}$ from $U(\mathbb M)^p$ and $\vec b\in\mathbb M^q$, the set
$\{ i\in\omega \ : \ \mathbb{M} \models
  \varphi(\vec g_i,\vec{b})\}$
is either finite or co-finite (in $\omega$),
\item for every formula $\varphi(x,\vec y)$, indiscernible sequence $(a_i)_{i\in\omega}$ from $\mathbb M$ and $\vec b\in \mathbb M^q$ with $a_i\notin\dcl_{\Cal L}(U(\mathbb M),\vec b)$ for every $i\in\omega$, the set
$\{ i\in\omega \ : \ \mathbb{M} \models
  \varphi({a_i},\vec{b})\}$
is either finite or co-finite (in $\omega$).
\end{enumerate}
Then $T_B$ is dependent.
\end{prop}
\begin{proof} Let $(a_i)_{i\in \omega}$ be an
indiscernible sequence, $\varphi(x,\vec{y})$ be an
$\Cal{L}(U)$-formula and
$\vec{b} \in \mathbb{M}^p$. We distinguish two cases.

\medskip\noindent
{Case I:} $\{a_i:i\in
\omega\}$ is $\dcl_{\Cal L}$-dependent over $U(\mathbb M)$.\\
As $(a_i)_{i\in\omega}$ is indiscernible, there is $i_0\in\omega$ such that $a_i\in\dcl_{\Cal L}(U(\mathbb M)\cup\{a_1,\dots,a_{i_0}\})$, for every $i>i_0$. Using Proposition \ref{indlemma}, take a function $f:\mathbb M^m\to\mathbb
M$ that is $\Cal{L}$-definable over $U(\mathbb
M)\cup\{a_1,\dots,a_{i_0}\}$ and an indiscernible sequence $(\vec
g_i)_{i>i_0}$ from $U(\mathbb M)^m$ such that $f(\vec g_i) = a_i$ for every
$i>i_0$.
Then by (1), the set
\begin{equation*}
\{i\in\omega \ : i>i_0\text{ and }\ \mathbb{M}\models
\varphi(f(\vec g_i),\vec{b}) \}
\end{equation*}
is finite or cofinite; hence so is the set
\begin{equation*}
\{ i \in \omega \ : \ \mathbb{M}\models \varphi(a_i,\vec{b}) \}.
\end{equation*}

\medskip\noindent
{Case II:} $\{a_i:i\in \omega\}$ is $\dcl_{\Cal L}$-independent over $U(\mathbb M)$.

\smallskip\noindent
Suppose there is an infinite set $S \subseteq \omega$ such that
for every $i \in S$, $a_i\in\dcl_{\Cal L}(U(\mathbb M),\vec{b})$. Then
\begin{equation*}
\rk\big(\{a_i:i\in\omega\}|U(\mathbb M)\big)\leq\rk\big(\{b_1,\dots,b_q\}|U(\mathbb M)\big).\end{equation*}
But this is impossible, since the first term is infinite and the second
term is finite.
Therefore we may assume that $a_i\notin\dcl_{\Cal L}(U(\mathbb M)\cup\{b_1,\dots,b_q\})$ for every $i\in\omega$ and thus we are done by using (2).
\end{proof}


\section{Dense case}\label{densepairs}
In this section $\Cal A, B, T_B$ are as in the Introduction, that is $\Cal A$ is an o-minimal structure in the language $\Cal L=\{<,\dots\}$, $B$ is a subset of $A$ and $T_B$ is the theory of $(\Cal A,B)$ in the language $\Cal L(U):=\Cal L\cup\{U\}$, where $U$ is a unary predicate not in $\Cal L$.

\medskip\noindent
Note that we do not assume that $B$ is dense in $A$. However, this section is called `{\it Dense case}' because in the applications of the main theorem below, $B$ is always dense in $A$ (see Subsections \ref{dense-subsection} and \ref{mann}).

 \begin{thm}\label{nipisoc} Suppose that for every model $(\Cal M,N)$ of $T_B$ the following hold:
\begin{itemize}
              \item [(i)] every subset of $N^n$ definable in $(\Cal M,N)$ is a boolean combination of sets of the
form $S\cap K$, where $S \subseteq M^n$
is definable in $\Cal M$ and $K\subseteq M^n$ is $\emptyset$-definable in $(\Cal M,N)$,
              \item [(ii)]  every subset of $M$ definable in $(\Cal M,N)$
 is a boolean combination of subsets of $M$ defined by
\begin{equation*}
\exists y_1 \cdots \exists y_q \ U(y_1) \wedge\dots \wedge U(y_q) \wedge
\varphi(x,y_1,\dots,y_q),
\end{equation*}
where $x,y_1,\dots,y_q$ are distinct variables and $\varphi(x,y_1,\dots,y_q)$ is a quantifier-free $\Cal{L}$-formula,
\item[(iii)] every open subset of $M$ definable in $(\Cal M,N)$ is a finite union of intervals.
\end{itemize}
Then $T_B$ is dependent.
\end{thm}
\begin{proof}
Let $\psi(\vec x,\vec y)$ be an $\Cal L(U)$-formula, $(\vec a_i)_{i\in\omega}$ an indiscernible sequence from $\mathbb M^p$, $\vec b\in\mathbb M^q$.
We use Proposition \ref{mainthm} to conclude that the set
$$J:=\{i\in\omega: \mathbb M\models \psi(\vec a_i,\vec b)\}$$
is either finite or cofinite.
First consider the case that $\vec a_i\in (U(\mathbb M))^p$ for every $i\in\omega$.

\medskip\noindent
By (i), we can assume that there are an $\Cal{L}$-formula $\varphi(\vec{x},\vec{y},\vec{z})$, an
$\Cal{L}(U)$-formula $\chi(\vec{x})$ without parameters and a tuple $\vec{c} \in
\mathbb{M}^l$ such that
\begin{equation*}
\mathbb{M} \models \psi(\vec a_i,\vec{b}) \textrm{ iff }
\mathbb{M} \models \varphi(\vec a_i,\vec{b},\vec{c})\wedge
\chi(\vec a_i),
\end{equation*}
for every $i\in\omega$. Since o-minimal theories are dependent and $(\vec a_i)_{i\in\omega}$ is an indiscernible sequence, we have that $J$ is finite or cofinite.

\medskip\noindent
In the rest of the proof we assume that $p=1$ and $a_i\notin\dcl\big(U(\mathbb M),\vec{b}\big)$ for every $n$. By (ii), we can assume that there is
an $\Cal{L}$-formula $\varphi(x,\vec{y})$ such that $\psi(x,\vec y)$
is equivalent to
\begin{equation*}
\exists z_1\cdots\exists z_m(\bigwedge_{i=1}^{m} U(z_i)\wedge
\varphi(x,\vec{y},z_1,\dots,z_m)).
\end{equation*}
For every $\vec{g} \in U(\mathbb M)^m$, the set
\begin{equation*}
X_{\vec{g}}:=\{ a\in\mathbb M \ : \ \mathbb{M} \models
\varphi(a,\vec{b},\vec{g})\}
\end{equation*}
is a finite union of intervals and points. Since each $a_n$ is $\dcl$-independent from $\vec{b},\vec{g}$, we have that $a_i\in X_{\vec{g}}$ iff $a_i \in \text{Int}(X_{\vec{g}})$, {\it the interior of} $X_{\vec{g}}$ . Set $$X:=\bigcup_{\vec{g} \in U(\mathbb M)^p} \text{Int}(X_{\vec{g}}).$$
Hence $X$ is open and for every $i\in\omega$
$$
\mathbb M \models \psi(a_i,\vec{b}) \text{ iff } a_i \in X.
$$
As $X$ is open, it is a finite
union of open intervals by (iii). Therefore $J=\{i:a_i\in X\}$ and so it is either finite or cofinite.
\end{proof}

\subsection{Dense pairs}\label{dense-subsection}
Here we observe that dense pairs of o-minimal structures as defined below are dependent.
In the setting we are interested in, these structures are defined and studied for the first time by van den Dries in \cite{densepairs}.

\medskip\noindent
Let $T$ be a complete o-minimal theory expanding the theory of ordered abelian groups with a distinguished positive element $1$. Let $\Cal L$ be the language of $T$ and $\Cal L(U)$ as before. A pair $(\Cal M,\Cal N)$ of models of $T$ is called a {\it dense pair} if $\Cal N\preceq \Cal M$, $M\neq N$ and $N$ is dense in $M$. Let $T_d$ be the theory of of such pairs $(\Cal M,\Cal N)$ in the language $\Cal L(U)$.
%
%
%
%
%
%

\medskip\noindent
By Theorem 2.5 of \cite{densepairs}, $T_d$ is complete; hence it equals $T_N$ (as defined in the Introduction) for any model $(\Cal M,\Cal N)$. Then the conditions (i), (ii) and (iii) of Theorem \ref{nipisoc} for $T_d$ are Theorems 1, 2 and 4 in \cite{densepairs}. So we get the following.
\begin{cor}\label{dependentdensepairs}
 The theory $T_d$ is dependent.
\end{cor}

\begin{remark} Assuming that $T$ has a distinguished positive element $1$ is not necessary for the conclusion of Corollary \ref{dependentdensepairs}. Let $\mathcal L'$ be the language $\mathcal L$ augmented by a constant symbol $c$ and let $T'$ be the extension of $T$ by the axiom $c>0$. Clearly, $T'_d$ is dependent if and only if $T_d$ is dependent. Therefore by Corollary \ref{dependentdensepairs}, $T'_d$ is dependent.
\end{remark}


\subsection{Groups with the Mann property}\label{mann}
In this part, $\Cal{L}$ is the language of ordered rings and $\R$ denotes both the (ordered) real field in that language and its underlying set.

\medskip\noindent
Let $\Gamma$ be a dense subgroup of $\R^{>0}$. We say that $\Gamma$ has the {\it Mann property} if for every $a_1,\dots,a_n\in\Q^\times$, there are only finitely many $(\gamma_1,\dots,\gamma_n)\in\Gamma^n$ such that $a_1\gamma_1+\cdots+a_n\gamma_n=1$ and $\sum_{i\in I}a_i\gamma_i\neq 0$ for every proper nonempty subset $I$ of $\{1,\dots,n\}$. Every multiplicative subgroup of finite rank in $\R^{>0}$ has the Mann property; see \cite{evertse-etal}.

\medskip\noindent
In the rest of this section we assume that $\Gamma$ has the Mann property and $\Gamma/\Gamma^{[p]}$ is finite for every prime $p$, where $\Gamma^{[p]}:=\{\gamma^p:\gamma\in\Gamma\}$. Let $\Cal L(U;\Gamma)$ be the language $\Cal L(U)$ augmented by a name for each element of $\Gamma$, and let $T(\Gamma)$ be the theory of $\big(\R,\Gamma,(\gamma)_{\gamma\in\Gamma}\big)$ in that language. Note that if $(R,G,(\gamma)_{\gamma\in\Gamma})$ is a model of $T(\Gamma)$, then $R$ contains a copy of $\Q(\Gamma)$ and $\Gamma$ is a pure subgroup of $G$.  From now on we denote models of $T(\Gamma)$ by $(R,G)$ rather than $\big(R,G,(\gamma)_{\gamma\in\Gamma}\big)$.

\medskip\noindent
Condition (ii) of Theorem \ref{nipisoc} is just Theorem 7.5 of \cite{vdDG} and condition (iii) 
follows directly from Lemmas 30 and 32 of \cite{BEG}.
%
%
%

\medskip\noindent
For a multiplicative group $G$, a tuple $\vec k=(k_1,\dots,k_n)$ of
integers, and $m>0$ let
$$G_{m,\vec k}:=\{(g_1,\dots,g_n)\in G^n:g_1^{k_1}\cdots g_n^{k_n}\in G^{[m]}\},$$
a subgroup of the group $G^n$.

\medskip\noindent
It follows from the assumption that $\Gamma/\Gamma^{[p]}$ is finite that $\Gamma_{m,\vec k}$ is of finite index in $\Gamma^n$ for every $m,n>0$ and $\vec k\in\Z^n$. Hence $G_{m,\vec k}$ is of finite index in $G^n$ for every $m,n>0$, $\vec k\in\Z$ whenever $(R,G)$ is a model of $T(\Gamma)$. Moreover in that case we may choose coset representatives for $G_{m,\vec k}$ in $G^n$ from $\Gamma^n$.
Now condition (i) of Theorem \ref{nipisoc} follows from the the next statement.

\begin{fact}(\cite{BEG}, Proposition 53)
 Let $(R,G)$ be a model of $T(\Gamma)$. A subset of $G^n$ definable in $(R,G)$ is a boolean combination of sets of the form $F\cap\vec\gamma G_{m,\vec k}$, where $F\subseteq R^n$ is definable in the ordered field $R$, $m>0$, $\vec k\in\Z^n$, and $\vec\gamma\in\Gamma^n$.
\end{fact}


\noindent
Therefore we get the desired result.
\begin{cor} The theory $T(\Gamma)$ is dependent.
\end{cor}

%

\begin{remark} Tychonievich showed in \cite{tychonievich} that for every subgroup $\Gamma$ of $\R^{>0}$ of finite rank, the expansion of $(\R,\Gamma)$ by the restriction of the exponential function to the unit interval defines $\mathbb Z$ and  hence is not dependent. However, in \cite{tau} proper o-minimal expansions $\mathcal R$ of $\R$ and finite rank subgroups $\Gamma$ are constructed such that the structure $(\mathcal R,\Gamma)$ satisfies the assumptions of Theorem \ref{nipisoc} and thus is dependent.
\end{remark}

\section{The discrete case}

Let $T$ be a complete o-minimal theory extending the theory of ordered abelian groups and let $\mathcal{L}$ be its language. After extending it by constants and by definitions, we may assume
that $T$ admits quantifier elimination and has universal axiomatization. Then any
substructure of any model of T is an elementary submodel. Hence $\dcl_{\Cal L}(X) = \langle X \rangle$ for any subset
$X$ of any model $\mathcal A$ of $T$; here $\langle X \rangle$ denotes the substructure of $\mathcal A$ generated by $X$. For $\mathcal B \preceq \mathcal A$ we denote $\langle B \cup X\rangle$ by $\mathcal B \langle X \rangle$.\\

\noindent In this section, we extend $\mathcal{L}$ to $\mathcal{L}(\mathfrak f)$ by adding a unary function
symbol $\mathfrak f$ which is not in $\Cal L$. Let $T(\mathfrak f)$ be a complete $\mathcal{L}(\mathfrak f)$-theory
extending $T$ and $\mathbb M$ a monster model of $T(\mathfrak f)$. For a model $(\mathcal A, \mathfrak f)$ of $T(\mathfrak f)$ and $X\subseteq A$, the $\Cal L(U)$-substructure of $(\Cal A,\mathfrak f)$
generated by $X$ is called the $\mathfrak f${\it -closure} of $X$ and denoted by $X^{\mathfrak f}$. Clearly $X^{\mathfrak f}\preceq \mathcal A$.




\begin{thm}\label{thmdiscpair}
Suppose that the following conditions hold.
\begin{itemize}
   \item [(i)] The theory $T(\mathfrak f)$ has quantifier elimination.
   \item [(ii)] For every $(\Cal A,\mathfrak f)\models T(\mathfrak f)$, $\Cal B\preceq
\Cal A$ with $\mathfrak f(B)\subseteq B$ and every $c_1,\dots,c_n\in A$, there are $d_1,\dots,d_n \in A$ such that
\begin{equation*}
 \mathfrak f\big(\Cal B\langle c_1,\dots,c_n\rangle\big) \subseteq  \langle\mathfrak f(B),d_1,\dots,d_n \rangle,
\end{equation*}
 \item [(iii)] Let $f,g$ be $\mathcal L$-terms of arities $m+k$ and $n+l$ respectively, $(\vec a_i)_{i\in \omega}$ an indiscernible sequence from $\mathbb M^m$
with $a_{i,1},\dots,a_{i,n}
\in \mathfrak f(\mathbb{M})$ for every $i\in\omega$, $\vec b_1 \in \mathbb{M}^k$ and
$\vec b_2 \in (\mathfrak f(\mathbb{M}))^l$. Then the set
\begin{equation*}
\big\{ i \in \omega \ : \ \mathbb{M} \models \mathfrak f(f(\vec
a_{i},\vec b_1))=g(a_{i,1},\dots,a_{i,n},\vec b_2)\big\}
\end{equation*}
 is finite or cofinite.
 \end{itemize}
Then $T(\mathfrak f)$ is dependent.
\end{thm}
%
%
%
\begin{proof}
By (i), we just need to show that for every quantifier-free
$\mathcal{L}(\mathfrak f)$-formula $\psi$, indiscernible sequence
$(\vec a_i)_{i\in\omega}$ and tuple $\vec{b}\in\mathbb{M}^m$
\begin{equation}\label{tamethmeq}J:=\{ i \in \omega \ : \ \mathbb{M} \models
\psi(\vec a_i,\vec{b})\} \textrm{ is  finite or cofinite.}
\end{equation}
We will prove this by induction on the number $d(\psi)$ of times $\mathfrak f$
occurs in $\psi$. If $d(\psi)=0$, this follows just from the fact that o-minimal theories are dependent. Suppose \eqref{tamethmeq} holds for all
quantifier-free $\mathcal{L}(\mathfrak f)$-formulas $\psi'$ with $d(\psi')<d$, and for a contradiction let $\psi$ be a quantifier-free
$\mathcal{L}(\mathfrak f)$-formula such that $d(\psi)=d$, $(\vec a_i)_{i\in\omega}$
an indiscernible sequence and $\vec{b}\in\mathbb{M}^m$ such that
\eqref{tamethmeq} does not hold for $\psi$. Then by Proposition \ref{shelahprop}, we can assume
that $(\vec a_i)_{i\in J}$ is an indiscernible sequence over
$\vec{b}$.  Since $d>0$ and $\psi$
is quantifier-free, there is an $\mathcal L$-term $f$ such that the term $\mathfrak f(f(\vec a_i,\vec{b}))$ occurs in
$\psi(\vec a_i,\vec{b})$.  Now let $\Cal A$ be the $\mathfrak f$-closure of $\{\vec a_i:i \in
J\}$. By (ii), there are $d_1,\dots,d_m\in \mathbb M$ such that
\begin{equation*}
\mathfrak f(\Cal A \langle \vec{b} \rangle) \subseteq \langle \mathfrak f(A),d_1,\dots,d_m
\rangle.
\end{equation*}
Then for every $j\in J$ we have
\begin{equation*}
\mathfrak f(f(\vec a_j,\vec{b})) \in \langle \mathfrak f(A),d_1,\dots,d_m \rangle.
\end{equation*}
Because $(\vec a_i)_{i\in J}$ is an indiscernible sequence over
$\vec{b}$, there exist natural numbers $k$ and $v$
with $1 \leq v \leq k$,  an $\mathcal L$-term $g$ and $\mathcal L(\mathfrak f)$-terms $t_1,\dots,t_n$ such that for every increasing sequence
$i_1<\dots<i_k$ of elements of $J$
\begin{equation}\label{caseIeq}
\mathfrak f(f(\vec a_{i_v},\vec{b}))=g(t_1(\vec a_{i_1},\dots,\vec a_{i_k}),\dots,t_n(\vec a_{i_1},\dots,\vec a_{i_k}),d_1,\dots,d_m).
\end{equation}
\noindent Take an infinite subset $\mathscr K$ of the set of $k$-element subsets of
$J$ and take an infinite subset $\mathscr L$ of the set of $k$-element
subsets of $\omega\setminus J$ such that for every $S,S'\in\mathscr K\cup\mathscr L$ either $s<s'$ for every $s\in S$, $s'\in S'$ or $s'<s$ for every $s\in S$, $s'\in S'$.
%
%
\noindent Since $(\vec a_i)_{i \in
\omega}$ is indiscernible, the sequence
\begin{equation*}
(\vec a_{i_1},\dots,\vec a_{i_k},t_1(\vec a_{i_1},\dots,\vec a_{i_k}),\dots,t_n(\vec a_{i_1},\dots,\vec a_{i_k}))_{\{i_1,\dots,i_k\}
\in \mathscr K \cup \mathscr L}
\end{equation*}
is indiscernible as well. By \eqref{caseIeq}, the equation
\begin{equation*}
\mathfrak f(f(\vec a_{i_v},\vec{b}))=g(t_1(\vec a_{i_1},\dots,\vec a_{i_k}),\dots,t_n(\vec a_{i_1},\dots,\vec a_{i_k}),d_1,\dots,d_m)
\end{equation*}
holds for infinitely many element of this sequence. Because of (iii), we
get that this equations actually holds for cofinitely many elements
of the sequence. By substituting \eqref{caseIeq} in $\psi$, we get a
quantifier-free $\mathcal{L}(\mathfrak f)$-formula $\psi'$ with $d(\psi')<d$ and
\begin{equation*}
\mathbb{M} \models \psi(\vec a_{i_v},\vec{b})\leftrightarrow\psi'\big(\vec a_{i_1},\dots,\vec a_{i_k},t_1(\vec a_{i_1},\dots,\vec a_{i_k}),\dots,t_n(\vec a_{i_1},\dots,\vec a_{i_k}),\vec{b},d_1,\dots,d_m\big)
\end{equation*}
 holds for cofinitely many $\{i_1,\dots,i_k\} \in \mathscr K
\cup \mathscr L$. But
\begin{equation*}
\mathbb{M} \models \psi(\vec a_{i_v},\vec{b}) \textrm{ iff }
\{i_1,\dots,i_v,\dots,i_k\} \in \mathscr K \textrm{for some }
i_1,\dots,i_k.
\end{equation*}
Hence the set of $\{i_1,\dots,i_k\} \in \mathscr K \cup \mathscr L$ such that
$$\mathbb{M}\models
\psi'\big(\vec a_{i_1},\dots,\vec a_{i_k},t_1(\vec a_{i_1},\dots,\vec a_{i_k}),\dots,t_n(\vec a_{i_1},\dots,\vec a_{i_k}),\vec{b},d_1,\dots,d_m\big)$$
is neither finite nor cofinite in $\mathscr K \cup \mathscr L$. This
contradicts the induction hypothesis and finishes the proof. \end{proof}


\section{Tame pairs}\label{tamepairs}

In this section, we consider tame pairs of o-minimal structures which were introduced by van den Dries and Lewenberg in \cite{tame}. Let $T$ be a complete o-minimal theory extending the theory of real closed
fields and $\mathcal{L}$ its language.  After extending $T$ by
definitions, we can assume without loss of generality that $T$ has
quantifier elimination and is universally axiomatizable. A pair $\Cal A,\Cal B$ of models of $T$ is called a {\it tame pair} if $\Cal B\preceq \Cal A$, $A\neq B$ and for every $a \in A$ which is in the convex hull of
$B$, there is a unique $\operatorname{st}(a) \in B$ such that $|a-\operatorname{st}(a)|<b$ for
all $b \in B^{>0}$. Note that the function $\operatorname{st}$ can be
extended to all $A$ by setting $\operatorname{st}(a)=0$ if $a$ is not in the
convex hull of $B$, and we call the resulting map as the {\it standard part map}. Note that the $\Cal L(U)$-structure $(\Cal A,B)$ is interdefinable with the $\Cal L(\operatorname{st})$-structure $(A,\operatorname{st})$. Now let $T_t$ be the $\Cal L(\operatorname{st})$-theory of tame pairs. Since $T$ has quantifier elimination, it follows from Theorem 5.9 and Corollary 5.10 of \cite{tame} that $T(\operatorname{st})$ is complete and has quantifier
elimination.


\medskip\noindent
The theory $T_t$ is closely related
to the theory $T_c$ of pairs $(\Cal A,V)$, where $\Cal A\models T$ and $V$ is
a $T$-convex subring of $A$ and $V \neq A$ (here and below a $T$-{\it convex} subring of a model $\Cal A$ of $T$ is a convex subring that is closed under all the continuous $\Cal L$-$\emptyset$-definable unary functions). By Corollary 3.14 of \cite{tame}, the theory $T_c$ is weakly o-minimal. By Proposition 7.3 of \cite{weakly}, every weakly o-minimal theory is dependent and hence so is $T_c$.

\medskip\noindent Note that for every
model $(\Cal A,\operatorname{st})$ of $T_t$, the pair $(\Cal A,V)$ is a model of $T_c$, where $V$ is
the convex closure of $\operatorname{st}(A)$. Since $V$ is a convex subring of $A$, it is a local ring. We denote its maximal ideal by $\mathfrak{m}(V)$.
In this case, for every $b \in \operatorname{st}(A)$ and $a
\in A$ \begin{equation}\label{convexeq}
 \operatorname{st}(a)=b \textrm{ iff } a=b \textrm{ or } (a-b \in \mathfrak{m}(V)) \textrm { or } (b=0 \textrm{ and } a \notin V).
\end{equation}


\begin{prop}\label{thmresineq} Let $\mathcal{C}\preceq\mathcal{D}\models T$ and let $d\in D\setminus C$. Let $W$ and $W'$ be $T$-convex subrings
of $\mathcal{C}$ and $\mathcal{C}\<d\>$ with
$d\in W'$ and $C\cap W'=W$. Then there is $d' \in W'$ such that for every $a\in W'$
$$
a - f(d',\vec{c}) \in \mathfrak{m}(W'),
$$
for some $\vec{c} \in W^n$ and $\mathcal L$-term $f$.
\end{prop}
\begin{proof} If for every $a\in W'$, there is $e\in W$ such that
$$
a - e \in \mathfrak{m}(W'),
$$
then the conclusion clearly holds. So suppose that it is not the case and take $d' \in W'$ such that
$$
d' - e \notin \mathfrak{m}(W'),
$$
for every $e \in W$. Then $d' \notin C$ since $W' \cap C=W$. Therefore by the Exchange principle $\mathcal C \<d\> = \mathcal C \<d'\>$ and hence $W'\subseteq \Cal C\<d'\>$. Now apply Lemma 5.3 of \cite{tame} (by taking $\mathscr R$, $V$, $V_a$ and $a$ there  to be $\Cal C$, $W$, $W'$ and $d'$) to conclude the proof.
\end{proof}

\medskip\noindent
Now we are in a position to prove the main result of this section.
\begin{thm}\label{tamedependent} $T_t$ is dependent.
\end{thm}
\begin{proof} We will show that $T_t$ satisfies the
assumptions of Theorem \ref{thmdiscpair}. As already mentioned above,
$T_t$ satisfies (i) by Theorem 5.9 from \cite{tame}.

\medskip\noindent
Now we  consider
(iii). Let $(\vec a_i)_{i\in \omega}$ be an indiscernible sequence from $\mathbb M^m$
such that there is a positive
$n\leq m$ with $a_{i,1},\dots,a_{i,n} \in \operatorname{st}(\mathbb{M})$ for every $i\in\omega$.
Also let $\vec b_1 \in \mathbb{M}^k$ and $\vec b_2 \in \operatorname{st}(\mathbb{M})^l$
and further $f,g$ be as in (iii) of Theorem \ref{thmdiscpair}. We now want to show that
\begin{equation*}
J:= \{ i \in \omega \ : \ \mathbb{M} \models \operatorname{st}(f(\vec
a_i,\vec b_1))=g(a_{i,1},\dots, a_{i,n},\vec b_2)\}
\end{equation*}
is finite or cofinite. Since $\operatorname{st}(\mathbb{M})$ is a model of $T$,
we have that for every $i\in\omega$
\begin{equation*}
g(a_{i,1},\dots ,a_{i,n},\vec b_2) \in \operatorname{st}(\mathbb{M}).
\end{equation*}
By \eqref{convexeq}, there is an $\Cal L(U)$-formula $\psi$ such that for every $i\in \omega$
$$T_t\models \operatorname{st}(f(\vec a_i,\vec b_1))=g(a_{i,1},\dots, a_{i,n},\vec b_2)\Longleftrightarrow T_c\models \psi(\vec{a_i},\vec{b_1},\vec{b_2}).$$
Since
$T_c$ is dependent, $J$ is finite or cofinite.

\medskip\noindent
It is left to show
(ii). So let $(\Cal A,\operatorname{st})\models T_t$, $\Cal B\preceq \Cal A$ and $a_1,\dots,a_n \in A$. Let $W_B$ be
the $T$-convex closure of $\operatorname{st}(B)$ in $B$ and $W_A$ the $T$-convex
closure of $\operatorname{st}(A)$ in $A$. Among the $n$-element sets $\{c_1,\dots,c_n\}$ with $\Cal B \<a_1,\dots,a_n\> = \Cal B \<c_1,\dots,c_n\>$ take the one with the maximal size of $\{c_1,\dots,c_n\} \cap W_A$. After renumbering, we have a natural number $l\leq n$ such that $c_i \in W_A$ iff $i \leq l$. We now show that
\begin{equation}\label{tameeq}
\Cal B \<c_1,\dots,c_n\> \cap W_A = \Cal B \< c_1,\dots,c_l \> \cap W_A.
\end{equation}
Towards a contradiction, let $c$ be in the left hand side but not in the right hand side of \eqref{tameeq}. Let $m$ be maximal with the property $c \notin \Cal B \< c_1,\dots,c_m \>$. Hence $l\leq m<n$ and
$$
c \in \Cal B\< c_1,\dots,c_{m+1} \> \setminus \Cal B \< c_1,\dots,c_m \>.
$$
By the Exchange principle, we have $\Cal B \< c_1,\dots,c_{m+1} \> =\Cal B \< c_1,\dots,c_{m},c \>$. Hence
$$
\Cal B \<a_1,\dots,a_n\> = \Cal B \< c_1,\dots,c_{n} \> = \Cal B \< c_1,\dots,c_{m},c,c_{m+2},\dots,c_n \>.
$$
But the cardinality of $\{c_1,\dots,c_{m},c,c_{m+2},\dots,c_n\}\cap W_A$ is at least $l+1$, contradicting the maximality of $l$.\\
\noindent Using Proposition \ref{thmresineq} inductively, we get $c_1',\dots,c_l'\in A$ such that for every $d \in \Cal B \<c_1,\dots,c_l\>$,
$$
d - f(c_1',\dots,c_l',\vec{b}) \in \mathfrak{m}(W_A).
$$
By \eqref{convexeq}, $\operatorname{st}(\Cal B \< c_1,\dots,c_l\>) \subseteq \<\operatorname{st}(B)\cup\{ \< c_1',\dots,c_l'\}\>$. Finally, by \eqref{tameeq}
$$
\operatorname{st}(\Cal B \< c_1,\dots,c_m\>) = \operatorname{st}(\Cal B \< c_1,\dots,c_l\>) \subseteq \<\operatorname{st}(B)\cup\{ \< c_1',\dots,c_l'\}\>.
$$
This establishes (ii) and finishes the proof.
\end{proof}


\section{Discrete groups}\label{discretegroups}

Let $\rtil$ be an o-minimal expansion of
$(\mathbb{R},<,+,\cdot,0,1)$ which is polynomially-bounded with
field of exponents $\mathbb{Q}$. Let $T$ be the theory of $\rtil$
and $\mathcal{L}$ be its language. We consider the structure
$(\rtil,2^{\mathbb{Z}})$. Since $2^{\mathbb{Z}}$ is discrete, we can
define a function $\lambda: \mathbb{R} \to
2^{\mathbb{Z}}\cup \{0\}$ by
\begin{equation*}
\lambda(x):=\left\{
              \begin{array}{ll}
                g, & \hbox{$x>0, g\in 2^{\mathbb{Z}}$ and $g \leq x < 2g$;} \\
                0, & \hbox{$x\leq 0$.}
              \end{array}
            \right.
\end{equation*}
Again, it is easy to see that the structures
$(\rtil,2^{\mathbb{Z}})$ and $(\rtil,\lambda)$ are interdefinable.
In \cite{miller2} generalizing the results from \cite{powersoftwo},
Miller showed that the latter has quantifier elimination up to
$\rtil$. For the following, we can assume that $\rtil$ has
quantifier elimination and has universal axiomatization.

\medskip\noindent
Let $T_{\operatorname{disc}}$ be the theory of
$(\rtil,\lambda)$ in the language $\mathcal{L}(\lambda)$, the
extension of $\mathcal{L}$ by a function symbol for the map
$\lambda$. As usual we let $\Cal L(U)=\Cal L\cup\{U\}$, where $U$ is a new unary predicate. For a model $(\Cal A,\lambda)$ of $T_{\operatorname{disc}}$, we sometimes want to refer to the $\Cal L(U)$-structure $(\Cal A,\lambda(A))$. In that case we put $G_{\Cal A}:=\lambda(A)\setminus \{0\}$.

\medskip\noindent
Let $\mathcal{A}$ be model of $T$, and define
\begin{equation*}
\fin{\Cal A}:= \{ x \in A : |x| \leq n \text{ for some }n>0\}.
\end{equation*}
It is easy to see that $\fin{\Cal A}$ is a local ring whose units are
\begin{equation*}
\Un{\Cal A}:= \{ x \in A : \frac{1}{n} \leq |x| \leq n \text{ for some }n>0\}.
\end{equation*}
\noindent Clearly $\fin(\Cal A)$ is convex and since the atomic model of $T$ is contained in $\mathbb R$, it is also $T$-convex by Proposition 4.2 in \cite{tame}. Let $v_{\Cal A} : A^{\times} \to \Gamma_{\Cal A}$ be the associated valuation: let $\Gamma_{\Cal A}:=A^\times/\Un{\Cal A}$ and $v_{\Cal A}(x):=x/\Un{\Cal A}$. Note that the multiplication of $\Cal A$ induces a group operation on $\Gamma_{\Cal A}$. Moreover, since $\Cal A$ is an expansion of a real closed field and hence closed under taking roots, the group $\Gamma_{\Cal A}$ is divisible and can be consider as a $\Q$-linear space with the scalar multiplication given by $q \cdot v(x) := v(x^q)$.

\medskip\noindent Let $\mathcal{B}$ be another model of $T$ with $\mathcal{B}\preceq \mathcal{A}$. It is easy to see that
$\fin{\Cal A} \cap \Cal B = \fin{\Cal B}$. Hence $\Gamma_{\mathcal{B}}$ can be embedded into $\Gamma_{\mathcal{A}}$. We will write $\dim_{\Q}(\Gamma_{\Cal A}/\Gamma_{\Cal B})$ for the $\Q$-linear dimension of the quotient space of $\Gamma_{\Cal A}$ and $\Gamma_{\Cal B}$.

\medskip\noindent
One of the key properties of polynomially-bounded o-minimal
structures is the {\it Valuation Inequality}.
We state below a particular
case; for the
general statement, see Corollary 5.6 in \cite{tameII}.

\begin{fact}[Valuation inequality-\cite{tameII}]\label{vieq}Let $\mathcal{A}$, $\mathcal{B}$ be models of $T$ with $\mathcal{B}\preceq \mathcal{A}$. Then
\begin{equation*}
\rk\big(\Cal A|\Cal B\big) \geq \dim_{\mathbb{Q}}(\Gamma_{\Cal A}/\Gamma_{\Cal B}),
\end{equation*}
where $\rk\big(\Cal A|\Cal B\big) := \inf \{ |X|  : X \subseteq A, \Cal B\langle X \rangle = \Cal A\}$.
\end{fact}

\noindent In the following, we establish several easy corollaries of the Valuation inequality.

\begin{cor}\label{corvieq} Let $\mathcal{A}$, $\mathcal{B}$ be models of $T$ with $\mathcal{B}\preceq \mathcal{A}$ and let
$a_1,\dots,a_m \in A$. Then there are $c_1,\dots,c_m \in
\mathcal{B}\langle a_1,\dots,a_m\rangle$ such that for every $d \in \mathcal{B}\langle
a_1,\dots,a_m\rangle$, there are $q_1,\dots,q_m \in \mathbb{Q}$ and $b \in \mathcal{B}$ with
\begin{equation*}
\frac{d}{b\cdot c_1^{q_1}\dots c_m^{q_m}} \in \Un{\Cal B\langle
a_1,\dots,a_m\rangle}.
\end{equation*}
\end{cor}

\begin{cor}\label{vic} Let $(\mathcal{B},\lambda) \preceq (\mathcal{A},\lambda) \models T_{\operatorname{disc}}$. Let $a_1,\dots,a_m \in A$. Then there are
$g_1,\dots,g_m \in G_{\Cal A}$ such that for every $d \in \mathcal{B}\langle
a_1,\dots,a_m\rangle$
\begin{equation*}
\lambda(d) = g \cdot g_1^{q_1}\dots g_m^{q_m}.
\end{equation*}
for some $q_1,\dots,q_m \in \mathbb{Q}$ and $g\in B\cap \lambda(A)$.
\end{cor}
\begin{proof} By Corollary \ref{corvieq}, there are $c_1,\dots,c_m \in \Cal B\langle a_1,\dots,a_m\rangle$ such that for every $d \in\Cal B\langle a_1,\dots,a_m\rangle$ there are $q_1,\dots,q_m \in \mathbb{Q}$, $b \in B$ and $n \in \N$ with
$$
\frac{1}{n} \leq \frac{d}{b\cdot c_1^{q_1}\dots c_m^{q_m}} \leq n.
$$
Set $g_0:= \lambda(b)$ and $g_i := \lambda(c_i)$ for $i=1,\dots,m$. Hence there is $n' \in \N$ such that
$$
\frac{1}{n'} \leq \frac{d}{g_0\cdot g_1^{q_1}\dots g_m^{q_m}} \leq n'.
$$
Then there is $l\in \Z$ such that
$$
1 \leq \frac{d}{2^l \cdot g_0\cdot g_1^{q_1}\dots g_m^{q_m}} < 2.
$$
Finally set $g:= 2^l\cdot g_0$. Then clearly $\lambda(d) = g \cdot g_1^{q_1}\dots g_m^{q_m}$.
\end{proof}

\medskip \noindent
We fix the following notation: Let $G$ be a torsion-free abelian group (written multiplicatively) and let $m>0$. Then
$$G^{[m]}:=\{g^m:g\in G\},$$
and for a pure subgroup  $H$ of $G$ and
$g_1,\dots,g_n\in G$, we define  $H_G\langle g_1,\dots,g_n\rangle$ to be
the smallest pure subgroup
\begin{equation*}
\big\{(h\cdot{g_1}^{k_1}\cdots{g_n}^{k_n})^{\frac{1}{m}} \ | \ h \in H,k_1,\dots,k_n \in
\mathbb{Z},m>0, h\cdot{g_1}^{k_1}\cdots{g_n}^{k_n}\in
G^{[m]}\big\}
\end{equation*}
of $G$ containing $H$ and $g_1,\dots,g_n$.

\begin{cor}\label{vic2} Let $(\mathcal{B},\lambda) \preceq (\mathcal{A},\lambda) \models T_{\operatorname{disc}}$, and let $g_1,\dots,g_m \in
G$.  Put $G:=G_{\Cal A}$ and $H:=G_{\Cal B}$. Then
$$\big(\Cal B\langle g_1,\dots,g_m \rangle,
H_{G}\langle g_1,\dots,g_m \rangle\big)\preceq (\Cal A,G).$$
\end{cor}
 \begin{proof}  Since $T_{\operatorname{disc}}$ has quantifier elimination in the language $\Cal L(\lambda)$, we only the need to show that
$$
\lambda\big(\Cal B \langle g_1,\dots,g_m \rangle \big)= H_{G}\langle g_1,\dots,g_m \rangle \cup \{0\}.
$$
It is easy to see that $H_{G}\langle g_1,\dots,g_m \rangle \cup \{0\}\subseteq \lambda\big(\Cal B \langle g_1,\dots,g_m \rangle \big)$.
For the other inclusion, by induction, we may assume that $m=1$. So let $g\in G$.

\medskip\noindent
{\bf Claim.} If $g\notin H$, then $v_{\Cal A}(g)\notin\Gamma_{\Cal B}$.

\smallskip\noindent
{\it Proof of the claim.} For a contradiction, suppose there is $b \in B$ and $n >0$ such that
$$
\frac{1}{n} \leq \frac{g}{b} \leq n.
$$
By replacing $b$ by $\lambda(b)$ and multiplying $b$ by an element of $2^{\Z}$, we have an $h \in H$ such that
$$
1 \leq \frac{g}{h} < 2.
$$
Since $2$ is the smallest element of $G$ larger than $1$ and $\frac{g}{h} \in G$, we have $g=h$. Hence $g \in H$. This is a contradiction against $g\in G \setminus H$, finishing the proof of the claim.

\medskip\noindent
We need to show that $\lambda(c)$ is in $H_G\langle g \rangle$ for every $c\in\Cal B\<g\>$. By Fact \ref{vieq}, the $\Q$-dimension of $\Gamma_{\Cal B\langle g \rangle}$ over $\Gamma_{\Cal B}$ is either $0$ or $1$. If it is $0$, then we are done, as $g\in H\subseteq B$. If that dimension is $1$, then using the claim, $\Gamma_{\Cal B\<g\>}$
is the $\Q$-linear subspace of $\Gamma_{\Cal A}$ generated by $\Gamma_{\Cal B}$ and $v_{\Cal A}(g)$.

\noindent Let $c \in \Cal B\langle g \rangle$. There are $q \in \Q$, $n\in \N$ and $b \in B$ such that
$$
\frac{1}{n} \leq \frac{c}{b\cdot g^q} \leq n.
$$
As above, we choose $b$ such that $b \in H$ and
$$
1 \leq \frac{c}{b\cdot g^q} < 2.
$$
Hence $\lambda(c) = b\cdot g^q$ and so $\lambda(c) \in H_G\langle g \rangle$.
 \end{proof}

\begin{thm}\label{discdependent} $T_{\operatorname{disc}}$ is dependent.
\end{thm}
\begin{proof} We just need to check that $T_{\operatorname{disc}}$ satisfies
the assumption of Theorem \ref{thmdiscpair}. Quantifier elimination
is shown in the proof of Theorem 3.4.2 in \cite{miller2}.
Assumption (ii) follows directly from Corollary \ref{vic}.

\medskip\noindent
So it is
only left to show (iii). Therefore let $\mathbb{M}$ be a
monster model of $T_{\operatorname{disc}}$ and take an indiscernible sequence
$(\vec a_{i})_{i\in \omega}$ from $\mathbb M^m$ such that there is $n\leq m$ with $a_{i,1},\dots,a_{i,n} \in \lambda(\mathbb{M})$ for every $i\in\omega$. Further
let $\vec b_1 \in \mathbb{M}^k$ and $\vec b_2 \in
\lambda(\mathbb{M})^l$. For a contradiction, suppose that there are $\mathcal L$-terms $f,g$ such that
\begin{equation*}
J:= \{ i \in \omega \ : \ \mathbb{M} \models
\lambda(f(\vec a_{i},\vec b_1))=g(a_{i,1},\dots,
a_{i,n},\vec b_2)\}
\end{equation*}
 is neither finite nor cofinite. Hence $J$ is an infinite subset of
\begin{equation*}
I:=\{ i \in \omega \ : \ \mathbb{M} \models g(a_{i,1},\dots a_{i,n},\vec b_2) \in \lambda(\mathbb M)\}.
\end{equation*}
By definition of $\lambda$, for every $i \in I$
$$
i \in J \textrm{ iff } \mathbb{M} \models 1 \leq
\frac{f(\vec a_{i},\vec b_1)}{g(a_{i,1},\dots,a_{i,n},\vec b_2)}<2.$$
Since $T$ is dependent, the right hand side must hold for cofinitely
many elements of $I$. Hence $J$ is cofinite in $I$. Therefore if $I$ is cofinite in $\omega$, then $J$ must be cofinite in $\omega$ as well.  Thus in order to show (iii) of Theorem
\ref{thmdiscpair} holds for $T_{\operatorname{disc}}$, it is only left to show that $I$ is cofinite in $\omega$.

\medskip\noindent
By Proposition \ref{shelahprop}, we may assume that $(\vec a_i)_{i\in I}$ is indiscernible over $\vec b_1,\vec b_2$.  By Corollary \ref{vic2}, for
every $i \in I$ there are $q_1,\dots,q_n\in \mathbb{Q}$ such that
\begin{equation*}
g(a_{i,1},\dots, a_{i,n},\vec b_2) =
a_{i,1}^{q_1}\dots a_{i,n}^{q_n}\cdot\vec b_2^{\,\vec{k}},
\end{equation*}
where $\vec k\in\Z^l$ and $\vec b_2^{\,\vec k}$ is short for $b_{2,1}^{k_1}\cdots b_{2,l}^{k_l}$.
Using once again the fact that $T$ is dependent, this equation holds not only for $i\in I$, but also for
cofinitely many $i\in \omega$. Hence we can assume that
$$
I = \{ i \in \omega : a_{i,1}^{q_1}\dots a_{i,n}^{q_n}\cdot\vec b_2^{\,\vec{k}} \in \lambda(\mathbb M)\}.
$$
\noindent It is left to show that for
cofinitely many $i\in \omega$
$$a_{i,1}^{q_1}\dots a_{i,n}^{q_n}\cdot\vec b_2^{\,\vec{k}} \in G_{\mathbb M}.$$
Let $M\in \N$ be such that $q_1\cdot M,\dots,q_n\cdot M
\in \mathbb{Z}$ and $M \cdot \vec{k} \in \mathbb{Z}^{l}$. So
we need to show that for cofinitely many $i\in \omega$, we have
\begin{equation}\label{thmdiscgroupeq}
a_{i,1}^{q_1\cdot M}\dots a_{i,n}^{q_n\cdot M} \in \vec b_2^{\,\vec{k}\cdot M} \cdot G_{\mathbb M}^{[M]}.
\end{equation}
Clearly $G_{\mathbb M}^{[M]}$ has only finitely many cosets in $G_{\mathbb M}$, since
$|2^{\mathbb{Z}}: (2^{\mathbb{Z}})^{[M]}|=M$. Further
$1,2,\dots,2^{M-1}$ are representatives of this cosets. Let $s \in
\{0,\dots,M-1\}$ be such that $\vec b_2^{\,\vec{k}\cdot M}$ is in $2^s \cdot
G_{\mathbb M}^{[M]}$. Then for every $i \in \omega$, we have that
\eqref{thmdiscgroupeq} holds iff
\begin{equation}\label{thmdiscgroupeq2}
a_{i,1}^{q_1}\dots a_{i,n}^{q_n} \in 2^s \cdot G_{\mathbb M}^{[M]}.
\end{equation}
Since \eqref{thmdiscgroupeq} holds for $i \in I$ and $(a_i)_{i\in
\omega}$ is an indiscernible sequence, the condition
\eqref{thmdiscgroupeq2} holds for all $i \in \omega$. Hence for
cofinitely many $i\in \omega$, we have that
$g(a_{i,1},\dots a_{i,n},\vec b_2) \in G_{\mathbb M}$. Hence $I$ is cofinite in $\omega$. \end{proof}

\begin{remark} For Theorem \ref{discdependent}, the assumption that the field of exponents of $\tilde{\mathbb R}$ is $\mathbb Q$ is necessary. If $\tilde{\mathbb R}$ defines an irrational power function, the structure $(\tilde{\mathbb R},2^{\mathbb{Z}})$ defines $\mathbb Z$ by  Corollary 1.5 of \cite{integers} and hence is not dependent.
\end{remark}


 \bibliography{ref}

\end{document}